\documentclass[a4paper,11pt]{article}
 \usepackage[T1]{fontenc} 
 \usepackage{a4wide}
 \usepackage{times}
\usepackage{amsthm, amsmath, amssymb}
\usepackage{algorithm}
\usepackage[noend]{algpseudocode}
\usepackage[british]{babel}
\usepackage{cases}
\usepackage{chngcntr}
\usepackage{mathrsfs}
\usepackage{hyperref}
\usepackage{natbib}
\usepackage{fancyvrb}
\usepackage{bbm}
\usepackage{bm}
\usepackage{latexsym}
\usepackage[title]{appendix}
\usepackage{color, graphicx}
\usepackage[normalem]{ulem}
\usepackage[font={small}]{caption}
\usepackage{lipsum}
\usepackage{float}
\usepackage{verbatim}
\usepackage{url}
\usepackage{authblk}

\makeatletter
\renewcommand\@biblabel[1]{-}
\makeatother
\theoremstyle{remark}
\newtheorem*{remark}{Remark}
\newtheorem*{xmpl}{Example}
\theoremstyle{plain}

\newtheorem{lemma}{Lemma}
\newtheorem{assumption}{Assumption}
\definecolor{jpred}{rgb}{0.99,0,0}

\definecolor{abblue}{rgb}{0,0,0.8}

\def\argmax{\mathop{\rm argmax}}
\def\argmin{\mathop{\rm argmin}}

\title{Regress-Later Monte Carlo for Optimal Inventory Control with applications in energy}

\author{Alessandro Balata\thanks{email: bn14a5b@leeds.ac.uk} }
\author{Jan Palczewski\thanks{email: J.Palczewski@leeds.ac.uk}}
\affil{School of Mathematics, University of Leeds, LS2 9JT, Leeds, United Kingdom}

\begin{document}

\maketitle

\begin{abstract}
We develop a Monte-Carlo based numerical method for solving discrete-time stochastic optimal control problems with inventory. These are optimal control problems in which the control affects only a deterministically evolving inventory process on a compact state space while the random underlying process manifests itself through the objective functional. We propose a Regress Later modification of the traditional Regression Monte Carlo which allows to decouple inventory levels in two successive time steps and to include in the basis functions of the regression the dependence on the inventory levels. We develop a backward construction of trajectories for the inventory which enables us to use policy iteration of Longstaff-Schwartz type avoiding nested simulations.
Our algorithm improves on the grid discretisation procedure largely used in literature and practice, and on the recently proposed control randomisation by \citet{Kharroubi2013}.  We validate our approach on three numerical examples: a benchmark problem of energy arbitrage used to compare different methods available in literature; a multi-dimensional problem of control of two connected water reservoirs; and  a high-dimensional problem of the management of a battery with the purpose of assisting the operations of a wind turbine in providing electricity to a group of buildings in a cost effective way.
\end{abstract}

\textbf{keywords:} stochastic optimal control;   regression Monte Carlo;   regress later;    inventory control;   energy arbitrage.

\section{Introduction}
\label{intro}

Our study of inventory control problems is motivated by the management of a limited amount of resources (which can be replenished or not) in order to maximise a measure of performance based on possibly both the inventory level itself and an exogenous process which for example could be the price or the demand for that particular resource. The inventory level influences the admissible controls since when the inventory is full it is not possible to accumulate resources further and analogously when it is empty it is not possible to withdraw them. Examples include management of energy storage such as batteries \citep{Haessig2013}, flywheels or pump-hydro stations for electricity \citep{Zhaoa2009}, underground caves for gas storage \citep{Boogert2008}; or management of warehouse/shop stocks subject to fluctuations of product prices and demand \citep{Secomandi2010}. 

In general, an inventory control problem involves two, possibly multi-dimensional, processes:  an \emph{underlying/exogenous process} being a discrete-time Markov process (Markov chain) with an arbitrary state space and a controlled \emph{inventory process} evolving deterministically on a compact state space. The set of admissible controls is usually state-dependent particularly due to the constraints on the inventory levels. The exogenous process is not affected by controls but influences the set of admissible controls and the objective functional. In an energy arbitrage context where the storage is used to buy energy cheap and sell expensive, the underlying process is the energy price while the inventory process monitors the level of the storage. More complex problems are discussed in Section \ref{numerics}.

We intend to approach our problem in a probabilistic way, through its dynamical programming formulation and, in particular, we study algorithms which can be classified as regression Monte Carlo methods. Introduced in literature by \citet{Carriere1996},\citet{Tsitsiklis2001} and \citet{Longstaff2001} these methods have proved their effectiveness in pricing of American and Bermudan options for complex dynamics of underlying instruments \citep{Jain2012}. They borrow from standard Monte Carlo methods their independence on the dynamics as long as these can be simulated. This feature prompted interest in applying regression Monte Carlo to inventory control problems where the exogenous process is commonly multi-dimensional and evolving on a non-compact state space preventing successful application of Markov chain approximation methods \citep{KD2001}. Although similarly to pricing of American/Bermudan options the control in inventory problems does not affect the random dynamics,  the action of the control changes the future inventory level making the problem significantly more difficult. The simplest and most commonly used approach borrows from Markov chain approximations by discretising the inventory space and applying regression Monte Carlo independently at each point of the discrete grid, see, e.g., \citet{Boogert2008, Moriarty2015, Delage2014,  Secomandi2014}. However, due to separate regressions needed at each discretisation point of the inventory, the method is inefficient and computational complexity becomes prohibitive when the dimension of the inventory space grows. %\ab{Moreover, our numerical study shows that the computed strategy is significantly inferior to those obtained by other methods available in the literature and the ones proposed in this paper}. 

A natural improvement to the inventory space discretisation is to include the inventory variable in the regression in order to be able to read from the resulting regression surface all information ``contained in each discretisation level'' and, without interpolating, those in-between. Algorithms in \cite{Carmona2010, Denault2013} require that the set of available controls be finite and small (in their examples, they allow for up to three values of control: increase, decrease, hold). Particularly, the approach by \cite{Denault2013} is sensitive to the size of the control set as for each possible control value it runs a separate regression at each time step in the backward dynamic programming procedure. \citet{Kharroubi2013} treat the control as a state variable and approximate the continuation value in dynamic programming procedure as a function of three variables: the state of endogenous process, the inventory and the control. This allows them to optimise over an infinite (or even continuous) set of controls at the cost of adding additional dimensions to the regression and basis functions. \citet{Langrene2015} shows an application of this methodology to natural resource extraction. In fact, \citet{Kharroubi2013} solve a more general problem of controlling diffusions, so that the control affects the random dynamics itself.

This paper builds on a non-standard variation of regression Monte Carlo called Regress Later Monte Carlo (RLMC) which can be traced back to \citet{Broadie2000, Glasserman2002, Broadie2004} and was recently studied in \citet{Beutner, Jain2015, Secomandi2014}. In the classical regression Monte Carlo (Regress Now Monte Carlo, RNMC), the conditional expectation of future payoffs (continuation value) appearing in the dynamic programming equation is approximated by a combination of a finite number of basis functions of present state variables through minimisation of the least square error along simulated trajectories of the underlying process. RLMC approximates first the value function at the following time by a finite combination of basis functions chosen in such a way that their conditional expectations can be computed analytically. 	\citet{Beutner} demonstrate that this modification of RNMC has a potential for significant improvement in the convergence rate of computed approximations to the value function. \citet{Secomandi2014} compare RNMC and RLMC in the context of energy real options and find that the latter excels in valuation of the dual upper bound for the value function. 

In the existing literature RLMC is predominantly regarded as a tool to reduce the approximation error in traditional exogenous Regression Monte Carlo problems (e.g., American option pricing). To the best of our knowledge, its ground-breaking potential for problems with endogenous (controlled) state variables has not been recognised yet. Our \emph{first contribution} is a formal derivation of the Regress Later Monte Carlo method for inventory problems avoiding discretisation of the inventory space as well as control randomisation. 
%; \ab{ the reader is referred to a companion paper \cite{Balata2017} }, where a general theory for controlled Markov chains is developed. 

Policy iteration of \cite{Longstaff2001} (here called \emph{performance iteration} to distinguish it from classical policy iteration methods in control of Markov chains \citep{KD2001}) improved upon \cite{Tsitsiklis2001} by reducing propagation of the approximation error of the continuation value in the recursive formula for value function of an optimal stopping problem. In their algorithm, at each step of the dynamic programming backward procedure the continuation value is computed as a conditional expectation of realised payoffs obtained by following the estimate of the optimal strategy. This approach relies on the fact that the control (stopping) manifests itself through the objective functional but does not affect the dynamics of the underlying process. In inventory problems the level of inventory is determined by the control. Performance iteration would therefore require resimulation of inventory levels after an estimate of the optimal control is found in each backward step. The \emph{second contribution} of this paper is the introduction of a technique for construction of optimal paths for the inventory backward in time, which considerably reduces the computations required by the performance iteration approach. 

In the final part of the paper, we perform a numerical study of existing algorithms for inventory control.  For this we use a benchmark problem of energy arbitrage on a finite time horizon in which the electricity store is employed to buy energy when the energy price is low and to sell when the price is high. The exact strategy is intricate as it has to take into account the state of charge, the distribution of future prices and the time remaining for trading. We demonstrate that the algorithm involving discretisation of the inventory variable requires long computational time to achieve acceptable quality of estimated control policy. The regress-later method introduced in this paper computes good policies up to 50 times faster. The efficiency of control randomisation of \citet{Kharroubi2013} is much better than grid-discretisation but significantly worse than of our algorithm.
We further examine the performance of grid-discretisation and regress-later  algorithms on a  pumped hydro-reservoir control problem with two dimensional inventory and control. It is in multi-dimensional problems that the grid discretisation methods particularly underperform due to the exponential increase of the number of discretisation points with the dimension. We show that our regress-later algorithm is able to tackle such problems efficiently and is less affected by the growth of dimensionality. 

We conclude the numerical section of the paper with a multi-dimensional problem motivated by \citet{Powell2014}. It consists in controlling a battery within a system with a wind turbine, stochastic demand and bi-directional connection to the grid trading at stochastic exogenous prices. The objective is to maximise the expected profits from operating the battery and involves using the battery to store electricity when the wind supply exceeds demand or when the grid price is low, whereas providing power to the demand or selling it to the grid when the price is high. The underlying stochastic process is three dimensional modelling wind power, demand and electricity price. The control is effectively two-dimensional with the inventory representing the state of charge of the battery. We demonstrate that regress-later Monte Carlo methods developed in this paper performs well in this multi-dimensional problem and exploit the statistical information about complicated dynamics of the underlying processes in the computation of the control policy. We further compute the value of the battery showing the potential for practical applications in project valuation (c.f. \cite{Guthrie2009, Swindle2014}).

The paper is organised as follows. In Section 2 we introduce the inventory problem including a discussion of a link to a continuous-time model. In Section 3 we critically analyse existing regression Monte Carlo-type approaches for inventory problems with two alternative iterative methods: value and performance iteration.
% It is followed by an explanation why such methods cannot be directly applied to inventory problems and discussion of solution available in the literature. 
Section 4 contains our main contributions: a Regress Later Monte Carlo method for inventory problems and backward simulation of inventory process to avoid nested simulations within the context of performance iteration. In the last section we validate our approach on three numerical examples: a benchmark problem of energy arbitrage used to compare different methods available in literature, a multi-dimensional problem of control of two connected water reservoirs, and a high-dimensional problem of the management of a battery with the purpose of assisting the operations of a wind turbine in providing electricity to a group of buildings in a cost effective way.

\section{Inventory Control problem}

%This paper assumes \emph{discrete-time dynamics} in the control problem. %An example of discretisation of a continuous-time dynamics of inventory problem is presented in Section \ref{subsec:cont_problem} in Electronic Appendix.
Consider a multidimensional process $(X,I)$ with an exogenous (uncontrolled) $p$-dimen\-sional random component $X_n(\omega)$, $n \ge 0$,  and an endogenous $q$-dimensional controlled component $I_n$ with values in $[0,I^1_{\text{max}}]\times\dots[0,I^q_{\text{max}}]$. We assume that $(X_n)$ is a Markov process and the discrete-time dynamics of $I_n$ is given by
\begin{equation}
\label{eq:invdynamics}
I_{n+1} = \varphi (n, u_n, I_n),
\end{equation}
where $\varphi$ is a Borel-measurable function and $u_n$ is a $q'$-dimensional control. Denoting by $\mathcal{F}^X$ the filtration generated by the process $(X_n)$, we specify the set of admissible controls $\mathcal{U}$ as all $\mathcal{F}^X$-adapted processes $(u_t)$ satisfying $u_n=(u_n^1,\,\dots,\,u_n^{q'}) \in [\underline{u}^1,\overline{u}^1]\times \cdots \times [\underline{u}^{q'},\overline{u}^{q'}]$ for all $n$ and such that $I_n \in [0,I^1_{\max}] \times \cdots \times [0,I^{q'}_{\max}]$. Without loss of generality and for simplicity of presentation, we will assume in the theoretical part of the paper that the inventory and the controls are one-dimensional , i.e., $I_n \in [0, I_{\max}]$ and $u_n \in [\underline{u}, \overline{u}]$. Numerical examples in Section \ref{numerics} are performed for multi-dimensional problems. 

In this setting, we define a performance measure
\begin{equation}
\label{eq:dperformance}
J(n,X,I,u)=\sum_{s=n}^{N-1} f(s,X_s,I_s,u_s)+g(X_N,I_N),
\end{equation}
The value function is given by $V(n,x,i) = \sup_{(u_s) \in \mathcal{U}} \mathbb{E} [ J(n, X, I, u) | X_n = x, I_n = i]$. As the dynamics of $(I_n)$ are deterministic and the control is adapted, we can exploit the measurability of $I_{n+1}$ with respect to $(X_n,I_n)$ in the Bellman equation for $V$:
\begin{equation}
\label{eq:dynamicinventorytrick}
\begin{cases}
\displaystyle V(n,x,i)=\sup_{\substack{u \in [\underline{u}, \overline{u}],\\ \varphi(n, u, i) \in [0, I_{\max}]}} \Bigl\{ f(n,x,i,u)+\mathbb{E}\bigl[V\big(n+1,X_{n+1},\varphi(n,u,i)\big) |X_n=x\bigr]\Bigr\},&\\
V(N,x,i)=g(x,i).&
\end{cases}
\end{equation}
This property will be useful for numerical methods presented in Section \ref{later}.

When designing a numerical algorithm for the solution of the dynamic programming  problem \eqref{eq:dynamicinventorytrick} the most natural approach is that of computing $V(n, x, i)$ backward in time, starting from the known terminal condition.  The problem, however, lies in the computation of the conditional expectation in \eqref{eq:dynamicinventorytrick} as the future state depends on the choice of control $u$. In the classical regression Monte Carlo, this would require redoing the regression step for each choice of the feedback function $u_n(X_n, I_n)$. In the following section, we discuss classical regression Monte Carlo in details and shed further light on the aforementioned problem with its application to an optimal control problem. We then sketch existing methods for overcoming this difficulty paving the way to our contribution.

\section{Critical analysis of Regression MC for inventory optimisation}
\label{CRMC}

As explained in the introduction, Regression Monte Carlo (RMC) is one of the most flexible approaches to solve inventory control problems. It was introduced in \citet{Carriere1996} and made famous by \citet{Longstaff2001} in the context of American option pricing. The method is very powerful for three main reasons:
\begin{itemize}
\item the method is applicable for any Markov process $X$ as long as it can be simulated;
\item the complexity, excluding the regression step, scales linearly with the number of dimensions;
\item the regression step allows to compute the conditional expectations without having to run a nested MC simulations for each state.
\end{itemize}
% \begin{remark}
% The first two points come from the use of Monte Carlo simulations, while the last one is the result of the introduction of regression instead of resimulation. 
% Note that the regression step can potentially scale exponentially with the number of dimensions, but significantly slower than the grid discretisation methods.  
% \end{remark}
This is made possible because the RMC method approximates conditional expectations using a linear combination of basis functions of state variables. The choice of basis functions plays a crucial role in the accuracy and computational complexity of approximating conditional expectations. When multidimensional basis functions are constructed by multiplying basis functions for individual dimensions, the complexity of the regression step grows exponentially with the number of directions. This can, however, be overcome with a careful choice of multi-dimensional basis functions based on domain knowledge or preliminary numerical results. 

\subsection{Regression approximation of conditional expectation}
Recall that in the formal presentation of the method for inventory problems, we restrict attention to the one-dimensional case: a single inventory and a one-dimensional random state process $(X_n)$. Notice that the higher dimensional case is identical; in fact the dimension of $X$ and $I$ does not influence any of the arguments that follow while only in Section \ref{backpolicy} we require the dimension of $I$ to be one. Define $\mathcal{L}^K_{\phi}$ as the linear space generated by the set of (linearly independent) functions $\{\phi_k(\cdot,\cdot)\}_{k=1}^K:\mathbb{R}\times[0,I_{\text{max}}]\to\mathbb{R}$.
The dynamic programming iteration  \eqref{eq:dynamicinventorytrick} can be computed in the regression MC framework taking advantage of the cross sectional information comprised in the different paths and projecting the value function at time $n+1$ over $\mathcal{L}^K_{\phi}(n)$, where with this notation we mean the set of basis function evaluated at time $n$, i.e. $(\phi_k(X_n,I_n))_{k=1, \ldots, K}$. More generally, for any functional $h:(\mathbb{R} \times [0, I_{\max}])^{N-n} \to \mathbb{R}$ we can compute an approximation of the conditional expectation $\hat{\mathbb{E}}_{n,x,i} \big[h((X_t, I_t)_{t=n+1, \ldots, N})\big] \approx\mathbb{E}\big(h((X_t, I_t)_{t=n+1, \ldots, N})|X_n = x,I_n = i\big)$ as follows:
\begin{equation}
\label{eq:condappr}
\hat{\mathbb{E}}_{n,x,i} \Bigl[h\bigl((X_t, I_t)_{t=n+1, \ldots, N}\bigr)\Bigr] = \sum_{k=1}^K \alpha_k \phi_k(x, i),
\end{equation}
where coefficients $(\alpha_k)$ solve the optimisation problem
\begin{equation}
\label{eq:condappr1}
\mathbb{E}\biggl[\mathbb{E}\Bigl[h\bigl((X_t, I_t)_{t=n+1, \ldots, N}\bigr) \Big|X_n,I_n\Bigr] - \sum_{k=1}^K \alpha_k \phi_k(X_n, I_n) \biggr]^2 \to \min.
\end{equation}
In the above, the outer expectation averages over the distribution of $(X_n, I_n)$ and we intentionally suppress in the notation the dependence of $I_{n+1}$ on the control $u_n$ for the reasons that will become clear while we describe existing methods for dealing with the limitations of the regression MC for controlled processes. We only hint here that one of these methods involves introducing randomness to the dynamics of $(I_n)$. 

A practical implementation of the RMC method requires generation of a set of $M$ trajectories $(Z_n^m, Y_n^m)_{n,m=1}^{N,M}$ of the process $(X_n,I_n)_{n=1}^N$ to approximate expectations in \eqref{eq:condappr1}. Notice that, because the inventory is controlled, the generation of the $Y$'s is problematic, but assume for now that it is possible. Coefficients $(\alpha^n_k)_{k=1, \ldots, K}$ in the approximation \eqref{eq:condappr} of the conditional expectation 
$\mathbb{E}\big(h((X_t, I_t)_{t=n+1, \ldots, N})|X_n = x,I_n = i\big)$ are then computed as
\begin{equation}
\label{eq:coefficients_MC}
\alpha^{n}=\argmin_{\alpha\in\mathbb{R}^K}\Bigl\{\sum_{m=1}^M\Big[h((Z^m_t, Y^m_t)_{t=n+1, \ldots, N})-\sum_{k=1}^K\alpha_k\phi_k(Z^m_{n},Y^m_{n})\Big]^2\Bigr\}.
\end{equation}
This is a standard linear regression procedure where $M$ responses $\big\{h((Z^m_t, Y^m_t)_{t=n+1, \ldots, N})\big\}_{m=1}^M$ are regressed against the set of $K$ explanatory variables $\big\{\phi_k(Z^m_{n},Y^m_{n})\big\}_{m = 1}^{M}$ indexed by $k=1, \ldots, K$.

\subsection{Value and Performance iteration}
\label{value}

Value iteration for optimal stopping problems has been introduced in the literature on regression Monte Carlo by \citet{Tsitsiklis2001}. In the optimal control case, the update rule for the value function follows directly from equation \eqref{eq:dynamicinventorytrick}:
\begin{equation}
\label{eq:valueit}
V(n,x,i)=\sup_{\substack{u \in [\underline{u}, \overline{u}],\\ \varphi(n, u, i) \in [0, I_{\max}]}}\Bigl\{ f(n, x,i,u)+\hat{\mathbb{E}}_{n,x,i, u}\bigl[V(n+1,X_{n+1},I_{n+1})\bigr]\Bigr\},
\end{equation}
where the estimator $\hat{\mathbb{E}}$ introduced in \eqref{eq:condappr} has been extended to accommodate the control $u$ at time $n$; recall that this control acts in a deterministic way on $I_n$ only. The difficulties arising  in extending \citet{Tsitsiklis2001} approach to inventory problems are discussed in Section \ref{why}. In a nutshell, the future value of $I_{n+1}$ depends on the control applied at time $n$ therefore preventing  the computation of the conditional expectation estimator $\hat{\mathbb{E}}_{n,x,I_n}\bigl[V(\cdot ,I_{n+1})\bigr]$. 

Policy iteration for regression Monte Carlo was introduced in \citet{Longstaff2001} for optimal stopping problems; despite its name, this approach is different to the well known policy iteration algorithm for the control of Markov chains \citep{KD2001} in which a parametrised control is improved iteratively. To avoid misunderstanding, throughout the paper we will refer to the Longstaff-Schwartz policy iteration approach as \emph{performance iteration}. Indeed, in this approach the update rule for the value function is based on the dynamic programming equation for the performance measure, instead of that for the value function:
\begin{equation}
\label{eq:policyit}
V(n,x, i)= \mathbb{E}\Bigl[f(n, x,i,u^*_n)+J\big(n+1,(X_t, I_t, u^*_t)_{t=n+1, \ldots, N})\bigr) |X_n=x,I_{n+1}=\varphi(n, u, i)\Bigr],
\end{equation}
where the "realised profits" (performance measure) are given by a recurrence relation
\begin{equation}
\label{eq:lspolicy}
J(n,(X_t, I_t, u^*_t)_{t=n, \ldots, N})=f(X_n,I_n,u^*_n (X_n, I_n))+J(n+1,(X_t, I_t, u^*_t)_{t=n+1, \ldots, N})
\end{equation}
 with $J(N,(X_N,I_N, u^*_N))=g(X_N,I_N)$.
 For the simplicity of notation, in parameters of $J$ we include $u^*_N$ which is never used since there is no control exerted at the terminal time $N$; for the sake of clarity, we will put $u^*_N (x,i) = 0$. An optimal control $u^*_n$ is estimated via 
\begin{equation}
\label{eq:control}
u^*_n (x,i)=\argmax_{\substack{u \in [\underline{u}, \overline{u}],\\ \varphi(n, u, i) \in [0, I_{\max}]}}\Bigl\{ f(n, x,i,u)+\hat{\mathbb{E}}_{n,x,i, u}\bigl[J(n+1,(X_t, I_t, u^*_t)_{t=n+1, \ldots, N})\bigr]\Bigr\}.
\end{equation}

Note that a direct implementation of equation \eqref{eq:policyit} requires computation of 
\[
J(n+1,(X_t, I_t, u^*_t)_{t=n+1, \ldots, N})=\sum_{t=n+1}^{N-1} f(k,X_t,I_t,u^*_t)+g(X_N,I_N)
\] 
for each choice of $u^*_n$, which numerically means resimulating the path from time $n$ to the terminal time incurring heavy computations. \cite{Longstaff2001} in an optimal stopping context propose a more efficient update rule based on \eqref{eq:lspolicy} in the sense that, given a path of $(X,I)$, it is possible to directly use the value of $J(n+1,(X_t, I_t, u^*_t)_{t=n+1, \ldots, N})$ computed at the previous time step to update $J(n,(X_t, I_t, u^*_t)_{t=n, \ldots, N})$. In the control case, this however requires that the value of $J(n+1, \cdot)$ is known for the particular trajectory $(X_t, I_t, u^*_t)_{t=n+1, \ldots, N}$ which itself depends on the control $u^*_n$ applied at time $n$. This seemingly invalidates the improvement of \cite{Longstaff2001} that avoids resimulation. We show in Section \ref{later} how this can be overcome.

In many problems, given a good approximation of the conditional expectations, both value and performance iteration can achieve precise estimates. Sometimes however, either when the value function has a  ``difficult'' shape  or because the running profit per period is relatively small compared to the continuation value, it is better to use performance iteration in order to minimise errors and improve the quality of the estimated policy. This approach works very well because the estimates of the conditional expectations are used only to take decisions while the value of these decisions is computed exactly, forward in time. Despite being time consuming the performance iteration can sometimes be the only possible alternative, in particular when using polynomial basis function because, even though they could be good enough to take decisions, they may be insufficient for accurate representation of the value of those. It should be pointed out however that the empirical regression calculated in the performance iteration case can be unstable, suffering from the relatively high variance of the  training points. In the value iteration case, on the other hand, response points are typically close to the previously estimated value function, causing the propagation of the regression error over time and, therefore, leading to stable, but biased, regression coefficients.  

Finally note that both value and performance iteration are designed to compute an estimation of the optimal policy $(u_t^*)$ and, in fact the value function obtained as output of the backward iterations should not be used as an approximation of the performance of this policy. Rather, we run an additional forward Monte Carlo simulation to establish the value of the estimated policy. The only output of the backward procedure should therefore be the matrix of regression coefficients $A=\{\alpha_k^n\}_{n,k=1}^{N,K}$, to be used in  forward simulations, to take decisions. The value of the policy induced by $A$ is computed as
\begin{equation}
\label{eq:eval}
v(x,i)=\frac{1}{M'}\sum_{m=1}^{M'}\Bigl(\sum_{n=1}^N f(n,Z^m_n,Y^m_n,\nu^m_n)+g(Z^m_n,Y^m_N) \Bigr)\,,\,\,\,\
\end{equation}
where the inventory process evolves as $Y_{n+1}^m=\varphi(Z_n^m,Y_n^m,\nu_n^m)$ with $Z^m_1=x$, $Y^m_1=i$ and %the estimated optimal control is computed as:
\begin{equation}
\label{eq:controlpo}
\nu^m_n=\argmax_{\substack{u \in [\underline{u}, \overline{u}],\\ \varphi(n, u, Y_n^m) \in [0, I_{\max}]}}\Bigl\{ f(n, Z_n^m,Y_n^m,u)+\sum_{k=1}^K\alpha^n_k\phi_k(Z^m_{n},Y^m_{n})\Bigr\}.
\end{equation}
Note that since the process $(X)$ is uncontrolled  the whole set of trajectories $(Z_n^m)^{N,M'}_{n,m=1}$ in equation \eqref{eq:eval} can be precomputed. In the following, when presenting the algorithms, we use the notation ``Evaluate the policy'' to refer to the routine specified by equations \eqref{eq:eval}-\eqref{eq:controlpo}.

\begin{remark}
It can be easily shown that, because the estimated policy must be not better than optimal, it is true that:
\[
\mathbb{E}[v(x,i)]\le V(1,x,i)
\]
\end{remark}

\subsection{Limitations of classical RMC for inventory problems}
\label{why}

In practical implementations of value or performance iteration, the approximation of conditional expectation operator $\hat E_{n, x, i} (\cdot)$ is computed using Monte Carlo approach as in \eqref{eq:coefficients_MC}. In the value iteration case, see \eqref{eq:valueit}, this amounts to regress a set of response points $\mathcal{Y} = \big(V(n+1,Z_{n+1}^m,Y_{n+1}^m)\big)_{m=1}^M$ over the sets of observations $\mathcal{B}=\big(\mathcal{X}_1,\,\dots,\,\mathcal{X}_K\big)$, where $\mathcal{X}_k = \big(\phi_k(Z_n^m,Y_n^m)\big)_{m=1}^M$ for  $k=1,\ldots, K$, and $\big(Z_{t}^m, Y_t^m\big)_{m,t=1}^{M,N}$ is the set of simulated trajectories of $(X_t, I_t)_{t=1}^N$. However, contrary to the uncontrolled case studied by \cite{Tsitsiklis2001} and \cite{Longstaff2001}, the process $I$ depends on the control, so one cannot compute first the approximation of $\hat{\mathbb{E}}_{n,x,i}\bigl[V(n+1,X_{n+1},I_{n+1})\bigr]$ and then optimise over controls. For each choice of control, one gets a different value of $I_{n+1}$. This yields two problems:
\begin{itemize}
\item[a)] The calculation of an optimal control at time $n$ is in itself a fixed point problem: one has to find such a control $u^*_n$  that after doing a regression approximation of the conditional expectation with $I_{n+1}$ determined by this control, the maximiser in \eqref{eq:valueit} equals to $u^*_n$;
\item[b)] The value of $V(n+1,X_{n+1},I_{n+1})$ has to be established for multiple outcomes of $I_{n+1}$ achievable from $I_n$ when applying an admissible control $u_n$. This prevents usage of previously computed values of $V(n+1, X_{n+1}, \cdot)$.
\end{itemize}

In the following subsections we briefly review methods available in literature to address the above issues. Our contribution is presented in Section \ref{later}. 

\subsection{Grid discretisation of inventory levels}
\label{discretisation}

Consider a discretisation $\Lambda_L=\{\lambda_1=0,\dots,\lambda_L=I_{\text{max}}\}$ of the state space of the controlled inventory, commonly an uniform grid $\Lambda_L = \{ (l-1)\frac{I_{\max}}{L-1}:\ l = 1, \ldots, L \}$. Simulate $M$ trajectories $(Z_{t}^m)_{m, t=1}^{M,N}$ of the uncontrolled process $(X_t)_{t=1}^N$. 
Note that the use of the grid is effectively equivalent to having  $\{Y_n^{m,l}=\lambda_l\}_{m,l=1}^{M,L}$. In order to compute the continuation value in \eqref{eq:valueit} regress the set of responses  $\{V(n+1,Z^m_{n+1},\lambda_l)\}_{m=1}^{M}$ over $\{\phi_k(Z_n^m)\}_{m,k=1}^{M, K}$ for each $l = 1, \ldots, L$ obtaining a collection of coefficients $\{\alpha_n^{k,l}\}_{k, l =1}^{K, L}$ and approximations $\sum_{k=1}^K \alpha_n^{k, l} \phi_k (x)$ of ${\mathbb{E}}[V(n+1,X_{n+1},\lambda_l)|X_n=x]$. Since these approximations are for a fixed set of inventory levels, an interpolation between those levels is needed in order to run the backward procedure in dynamic programming equation \eqref{eq:valueit}. Denoting the resulting approximation of the conditional expectation $\mathbb{E}[V(n+1,X_{n+1},i)|X_n=x]$ by $\mathcal{E} (n, x,i)$, we obtain an approximate backward dynamic programming equation
\[
V(n,Z_n^m,\lambda_l)=\sup_{\substack{u \in [\underline{u}, \overline{u}],\\ \varphi(n, u_n, \lambda_l) \in [0, I_{\max}]}} \Bigl\{f(n, Z_n^m,\lambda_l,u)+\mathcal{E}(n, Z_{n}^m,\varphi(n, u, \lambda_l))\bigr]\Bigr\}
\]
for $l = 1, \ldots, L$.  Algorithm \ref{alg:1} provides details of the implementation. Extension to performance iteration results in a minor amendment as can be seen in the algorithm.
\begin{algorithm}[tb]
 \caption{Regression Monte Carlo algorithm with Grid Discretisation (GD)}
 \label{alg:1}
 \begin{algorithmic}[1]
\State Generate training points $\{Z_{n}^{m}\}_{n=1,m=1}^{N,M}$
\State Initialise the value function $V(N, Z_{N}^{m},\lambda_l)=g(Z_{N}^{m},\lambda_l), \quad  j=1,\ldots,L, \ m = 1, \ldots, M$
 \For{$n=N-1$ to $1$}
\For{$l=1$ to $L$}
  \State $\alpha_n^l=\argmin\limits_{\alpha\in\mathbb{R}^K}\{\sum_{m=1}^M[V(n+1,Z_{n+1}^{m},\lambda_l)-\sum_{k=1}^K\alpha_k\phi_k(Z_{n}^{m})]^2\}$
\State $\hat{\mathcal{E}}(Z_{n}^{m},\lambda_l)=\sum_{k=1}^K \alpha_n^{k,l}\phi_k(Z_{n}^{m}),\quad \forall m$
\EndFor
\State Extend $\hat{\mathcal{E}}(Z_{n}^{ m}, i)$ to $i \in [0, I_{\max}]$ by interpolation 
\State $u_{n,m}^l=\argmax_{\substack{u \in [\underline{u}, \overline{u}],\\ \varphi(n, u_n, \lambda_l) \in [0, I_{\max}]}}\{f(n,Z_{n}^{m},\lambda_l,u)+\hat{\mathcal{E}} (Z_{n}^{m},\varphi(n, u, \lambda_l))\}, \quad \forall m,l$
  \If{value iteration}
\State $V(n,Z_{n}^{m},\lambda_l)=f(n,Z_{n}^{m},\lambda_l,u_{n,m}^l)+\hat{\mathcal{E}} (Z_{n}^{m},\varphi(n, u^l_{n,m}, \lambda_l))\}, \quad \forall m,l$
  \EndIf
  \If{performance iteration}
\State For each $m$ and $l$, compute inventory trajectories $\{Y_{j}^{m,l} \}_{j=n}^N$ along $(Z_j^m)_{j=n}^N$ starting from $Y_n^{m,l} = \lambda_l$ using control
\State $\displaystyle u^*_{n} (z, y)=\argmax_{\substack{u \in [\underline{u}, \overline{u}],\\ \varphi(n, u_n, \lambda_l) \in [0, I_{\max}]}}\{f(n,z,y,u)+\hat{\mathcal{E}} (z,\varphi(n, u, y))\}.$
\State Set
\State  $V(n,Z_{n}^{m},\lambda_l)=\sum_{j=n}^{N-1} f\left(j,Z_{j}^{m},Y_{j}^{m, l},u_n^*\left(Z_{j}^{m}, Y_{j}^{m,l}\right)\right)+g(Z_{N}^{m},Y_{N}^{m,l})$
% \State \ab{extend $V(n+1, Z_{n+1}^{m},\lambda^l)$ to $\underline{V}(n+1, Z_{n+1}^{m},i), i \in [0, I_{\max}],$ by interpolation}
% \State $V(n,Z_{n}^{m},\lambda_l)=f(n,Z_{n}^{m},\lambda_l,u^l_{n,m})+\ab{\underline{V}(n+1,Z_{n+1}^{m},\varphi(n, u^l_{n,m}, \lambda_l)),} \quad \forall m,l$
  \EndIf
  \EndFor
\State \textbf{Evaluate the policy}
\end{algorithmic}
\end{algorithm}

The inventory discretisation method is wide-spread because of the relative simplicity of extending a traditional Regression Monte Carlo algorithm to controlled dynamics, see, e.g., \citet{Boogert2008, Delage2014, Moriarty2015, Secomandi2014}. However, due to separate regressions needed at each discretisation point of the endogenous (controlled) component,  it could be prohibitive to implement it if there are more than two endogenous dimensions, see Section \ref{subsec:2D} for performance on a 2D optimisation problem.

A natural improvement to the discretisation technique would be that of including the inventory in the regression in order to be able to read from the resulting regression surface all information ``contained in each discretisation level'' and, without interpolating, those in-between. However, as discussed  in Section \ref{why}, this is not easily achievable and only recently some alternatives have appeared in the literature, see \citet{Carmona2010, Denault2013, Kharroubi2013, Langrene2015}.% shows an application of this methodology to natural resource extraction (i.e., with $\bar{u}=0$ in our notation).

In the following subsections we discuss methodological developments of the papers mentioned in the previous paragraph with the main focus on how the authors solve difficulties highlighted in Section \ref{why}.

\subsection{Quasi-simulation of the inventory}

In order to include the inventory in the regression, its trajectories $(I_n)$ have to be simulated before the control is computed. However, as the control directly guides the evolution of $(I_n)$, the inventory should rather be stated at time $N$ and then its paths constructed backward in time as the dynamic programming procedure progresses. \citet{Carmona2010} and \citet{Denault2013} resolved this issue in similar ways in the framework of an optimal switching problem related to the management of an energy storage device for which only 3 controls are available: charge, discharge and store; i.e. $u\in\{-1,0,1\}$. The small and finite size of the set of possible control is paramount for these methods. However, in our opinion, certain aspects of \citet{Carmona2010}'s approach can be extended to general controls.

In \citet{Carmona2010} the authors deal for the first time with the task of regressing $\{V(n+1,Z_{n+1}^m,Y_{n+1}^m)\}_{m=1}^M$ over both exogenous and inventory processes. Clearly only three predecessors $Y^m_n$ are possible and the authors suggest to test separately whether $Y^m_{n+1}$ is an optimal successor of one of these three. This is achieved by regressing the value function at time $n+1$ over observations $(Z^m_n,Y^m_{n+1})_{m=1}^M$ and later testing whether an optimal control at time $n$ for any of the three predecessors of $Y^m_{n+1}$ drives the inventory process to $Y_{n+1}^m$. If there is such a predecessor in the interval $[0, I_{\max}]$, the trajectory $Y^m$ is extended to time $n$ and the performance iteration is used via a direct update as explained in Section \ref{value}. Otherwise, the $Y^m_n$ is randomly placed in $[0,I_{\max}]$ and a value iteration approximation is used.

The approach in \citet{Denault2013} to construct backward paths of the inventory process uses  basis functions dependent on $X_n$ and $I_n$ and requires the computation of 3 separate regressions for each choice of control at step $n$. This procedure does not guarantee that the trajectories $(Y_n^m)_{n,m}^{N,M}$ will remain inside $[0,I_{\max}]$ and they decided therefore to introduce an artificial penalty term and a random replacement of the point $Y_n^m$ once it is too far from the interval $[0,I_{\max}]$.

Reliable and thorough tests are not available, however, our experience and the literature seem to suggest that the grid discretisation method is preferred over quasi-simulation whenever the dimensionality of the problem is small enough to allow its use.
 
\subsection{Control Randomisation}

Control randomisation was  introduced in \citet{Kharroubi2013}; it differs from the previous methods in that the control becomes a state variable and it is simulated (according to some arbitrary model that ensures admissibility) along trajectories of $(X_t, I_t)_{t=1}^N$. Data underlying the regression Monte Carlo consists of three sets of trajectories $\{ Z_n^{m},Y_n^{m},\tilde{u}_n^m\}_{n,m=1}^{N,M}$. In the case of value iteration \eqref{eq:valueit}, $\{V(n+1,Z_{n+1}^{m},Y_{n+1}^m) \}_{m=1}^M$ is regressed against basis functions evaluated at the points $\{Z_n^{m},Y_n^{m},\tilde{u}_n^m \}_{m=1}^M$. These regression basis functions are dependent now on the random control $\tilde{u}_n$, in addition to $X_n$ and $I_n$ so that the estimated continuation value will depend on the choice of the control (which is different on each sample trajectory). An optimal control at time $n$ given $X_n = Z_n^m$ and $I_n = Y_n^m$ is approximated by the expression (if the maximum is attained)
\begin{equation}\label{eqn:CR_opt}
u_n^*(x,i) = \argmax_{\substack{u \in [\underline{u}, \overline{u}],\\ \varphi(n, u, i) \in [0, I_{\max}]}}\Bigl\{ f(n, x,i,u)+\hat{\mathbb{E}}_{n,x,i, u}\bigl[V(n+1,X_{n+1},  I_{n+1}) )\bigr]\Bigr\},
\end{equation}
where, with a slight abuse of notation, we included  in the approximate conditional expectation $\hat{\mathbb{E}}$ the dependence on the control $u$. In general, multiple runs of the method could be needed to obtain precise estimates because the initial choice of the dummy control could drive the training points far from where the optimal control would have directed them.

\begin{remark}
In many practical applications, the mapping $u \mapsto \varphi(x, i, u)$ is injective (strictly increasing), hence there is a one-to-one correspondence between $u_n$ and $I_{n+1}$ given $I_n$. The addition of this new coordinate in the regression may be replaced by regressing $\{V(n+1,Z_{n+1}^{m},Y_{n+1}^m) \}_{m=1}^M$ against basis functions evaluated at points $\{Z_n^{m},Y_n^{m}, Y_{n+1}^m \}_{m=1}^M$. This shows that trajectories $\{Y_n^m\}_{n,m=1}^{N,M}$ may be constructed deterministically to offer maximum coverage of those parts of the state space in which the precision is the most desirable.
\end{remark}

For further reference see \citet{Langrene2015} where the authors apply the numerical method in the context of natural resource extraction. For details about the implementation and extension to the performance iteration case refer to Algorithm \ref{alg:2}. Notice that a resimulation stage, lines 9 and 10, is required in the performance iteration algorithm.
\begin{algorithm}[tb]
 \caption{Regression Monte Carlo algorithm: Control Randomisation (CR)\label{alg:2}}
 \begin{algorithmic}[1]
\State Generate a random control $\{\tilde{u}_{n}^{m}\}_{n=1,m=1}^{N,M}\in\mathcal{U}$ along the training points $\{Z_{n}^{m},Y_{n}^{m}\}_{n=1,m=1}^{N,M}$.
%\State Generate training points $\{Z_{n}^{m},Y_{n}^{m}\}_{n=1,m=1}^{N,M}$ according to the dynamics induced by $\tilde{u}$
\State Initialise the value function $V (N, Z_{N}^{m},Y_{N}^{m}) =g(Z_{N}^{m},Y_{N}^{m}), \quad \forall m$
 \For{n=N-1 to 1}
\State $\alpha_{n+1}=\argmin\limits_{\alpha\in\mathbb{R}^K}\{\sum_{m=1}^M [V(n+1,Z_{n+1}^{m},Y_{n+1}^{m})-\sum_{k=1}^K\alpha^k_{n+1}\phi_k(Z_{n}^{m},Y_{n}^{m},\tilde{u}_{n,m})]^2\}$ 
%\State Denote $\mathcal{E}(n, Z_{n,m},Y_{n,m},u)=\sum_{k=1}^K\alpha_{n+1}^k\phi_k(Z_{n,m},Y_{n,m},u)$,\quad $\forall m$
 \If{value iteration}
  \State For all $m$ do
\begin{flalign}\label{eqn:CR_alg_vi}
&V(n,Z_{n}^{m},Y_{n}^{m})\\ 
&=\sup_{\substack{u \in [\underline{u}, \overline{u}],\\ \varphi(n, u, Y_{n}^{ m}) \in [0, I_{\max}]}} \Big\{f(n,Z_{n}^{m},Y_{n}^{m},u)+\sum_{k=1}^K\alpha_{n+1}^k \phi_k(Z_{n}^{m}, Y_{n}^{m},u)\Big\} \notag
\end{flalign}
 \EndIf
 \If{performance iteration}
\State For each $m$, update trajectories $\{Y_{j}^{m} \}_{j=n+1}^N$ using control given by Equation \eqref{eqn:CR_opt}
\State  $V(n,Z_{n}^{m},Y_{n}^{m})=\sum_{j=n}^{N-1} f\left(j,Z_{j}^{m},Y_{j}^{m},u_n^*\left(Z_{j}^{m}, Y_{j}^{m}\right)\right)+g(Z_{N}^{m},Y_{N}^{m})$
\EndIf
\EndFor
\State \textbf{Evaluate the policy}
\end{algorithmic}
 \end{algorithm}

\section{Regress Later Monte Carlo: a decoupling approach}
\label{later}

In the previous section the regression was used to approximate the conditional expectation with respect to $(X_n, I_n)$ directly as a linear combination of basis functions of those two variables. As opposed to this \emph{regress-now} approach, we employ in this section a \emph{regress-later} idea in which the conditional expectation with respect to $(X_n, I_n)$ is computed in two stages. First, a conditional expectation with respect to $(X_{n+1}, I_{n+1})$ is approximated in a regression step by a linear combination of basis functions of $(X_{n+1}, I_{n+1})$. Then, analytical formulas (or precomputed numerical approximations) are applied to condition this linear combination of functions of future values on present values $(X_n, I_n)$.

\subsection{Regress later approximation}

Unlike the regress-now method for approximating conditional expectations, the regress-later approach imposes conditions on basis functions:
\begin{assumption}
\label{laterass}
For each basis function $\phi_k$, $k=1, \ldots, K$, the conditional expectation
\[
\hat{\phi}_k^n(x,i) = \mathbb{E}[\phi_k(X_{n+1}, i)|X_n=x]
\] 
can be computed analytically.
\end{assumption}

\begin{remark} $\ $
\begin{enumerate}
\item Assumption \ref{laterass} can be relaxed in practical numerical implementations because the conditional expectations of the basis functions can be numerically precomputed.
\item In many problems, the process $(X_n)$ is time homogeneous and, therefore, functions $\hat\phi^n_k$ do not depend on $n$. For simplicity of notation we drop the superscript in the following.
\item Multidimensional basis functions are often constructed as products of one-dimensional functions: $\phi_k(x, i) = \phi^X_k(x) \phi^I_k(i)$. Then the function $\hat\phi_k$ factors as follows:
\[
\hat{\phi}_k(x,i) = \phi^I_k(i)\, \mathbb{E}[\phi^X_k(X_{n+1})|X_n=x].
\]
\end{enumerate}
\end{remark}

We will now present regress-later solution to \emph{value iteration} procedure \eqref{eq:valueit}. Assume that at time $n+1$ the value function $V(n+1,\cdot)$ has been computed for a set of training points $\{Z_{n+1}^m,Y_{n+1}^m\}_{m=1}^{M}$. We perform a regression to approximate $V(n+1, \cdot)$ with a linear combination of basis functions:
\[
V(n+1,x,i) \approx \sum_{k=1}^K\alpha^{n+1}_k\phi_k(x,i),
\]
where
\begin{equation}
\label{eq:coeff}
\alpha^{n+1}=\argmin_{\alpha\in\mathbb{R}^K}\bigg\{\sum_{m=1}^M\Bigl[V(n+1,Z^m_{n+1},Y^m_{n+1})-\sum_{k=1}^K\alpha_k\phi_k(Z^m_{n+1},Y^m_{n+1})\Bigr]^2\bigg\}.
\end{equation} 

Moving one step backward to time $n$, we would like to compute the value function $V(n, \cdot, \cdot)$ and an optimal control $u_n(\cdot, \cdot)$. We select a set of training points $\{Z_n^m,Y_{n}^m\}_{m=1}^{M}$, which can be placed independently from  $\{Z_{n+1}^m,Y_{n+1}^m\}_{m=1}^{M}$, therefore, removing the limitation of the regress-now approach for which points at different times must be linked through their true dynamics. Here, these training points are only needed for further computation of an approximation of the value function at time $n$ with a linear combination of basis functions. 
An approximation of an optimal control is computed as (assuming that the maximum is attained)
\[
u^*_n (x, i) =\argmax_{\substack{u \in [\underline{u}, \overline{u}],\\ \varphi(n, u, i) \in [0, I_{\max}]}}\Bigl\{ f(n, x , i,u)+\sum_{k=1}^K\alpha^{n+1}_k \hat \phi^n_k \big(x, \varphi(n,u,i)\big)\Bigr\}
\]
with the value function at sampling points estimated by
\[
V(n, Z_n^m, Y_n^m) = f(n, Z_n^m,Y_n^m,u_n^m)+\sum_{k=1}^K\alpha^{n+1}_k \hat \phi^n_k \big(Z_n^m, \varphi(n,u_n^m,Y_n^m)\big)
\]
with $u_{n}^{m} = u^*_n (Z_n^m, Y_n^m)$.

The above procedure can be easily adapted to \emph{performance iteration} approach, \eqref{eq:policyit}, with resimulation. Assume that a  control mapping $u^*_j: \mathbb{R} \times [0, I_{\max}] \to \mathbb{R}$ is known for $j = n+1, \ldots, N-1$. Assume further that $(Y_{j}^m)_{j=n+1, \ldots, N}$ correspond to those controls, i.e., $Y_{j+1}^m = \varphi\big(j, u^*_j (Z_j^m, Y_j^m), Y_j^m\big)$ for $j=n+1, \ldots, N-1$.
%Assume further that $(Z_{j}^m, Y_{j}^m, u_j^m)_{j=n+1, \ldots, N}$ are sampled from the controlled dynamics with the initial condition at time $n+1$ being $(X_{n+1}, I_{n+1})  = (Z_{n+1}^m, Y_{n+1}^m)$ for each $m$. 
We approximate the conditional expectation $\mathbb{E} \big[J(n+1,(X_{j}, I_{j}, u^*_j)_{j=n+1, \ldots, N}) \big | X_{n+1} = x, I_{n+1} = i\big]$ by a linear combination of basis functions $\sum_{k=1}^K\alpha^{n+1}_k\phi_k(x, i)$ performing a least-squares regression:
\begin{equation}
\label{eq:coeff_pi}
\alpha^{n+1}=\argmin_{\alpha\in\mathbb{R}^K}\bigg\{\sum_{m=1}^M\Bigl[J\big(n+1,(Z_{j}^m, Y_{j}^m, u^*_j)_{j=n+1, \ldots, N}\big)-\sum_{k=1}^K\alpha_k\phi_k(Z^m_{n+1},Y^m_{n+1})\Bigr]^2\bigg\}.
\end{equation} 
Note that although $J(n+1, \cdot)$ is not itself measurable with respect to $(X_{n+1}, I_{n+1})$, as is the case for the value iteration algorithm, the least-squares regression in \eqref{eq:coeff_pi} approximates the aforementioned conditional expectation. Indeed, the above regression projects $J(n+1, \cdot)$ on the linear space spanned by functions $\phi_k(X_{n+1}, I_{n+1})$, $k=1, \ldots, K$. This can be viewed as a superposition of two projections: the first one on the space of generated by functions of $(X_{n+1}, I_{n+1})$ and further on its linear subspace generated by basis functions $\phi_1, \ldots, \phi_K$.

A backward step of the performance iteration algorithm computes a control mapping $u^*_n$:
\begin{equation}\label{eqn:back_cont}
u^*_n (x,i)=\argmax_{\substack{u \in [\underline{u}, \overline{u}],\\ \varphi(n, u, i) \in [0, I_{\max}]}}  \Bigl\{ f(n, x, i, u) +\sum_{k=1}^K\alpha^{n+1}_k \hat \phi^{n}_k \big(x, \varphi(n,u, i) \big)\Bigr\}.
\end{equation}
For all points $(Z_n^m, Y_n^m)$ the trajectories $\{ Y_j^m\}_{j=n+1}^N$ must be recomputed as it is unlikely that the inventory process governed by $u^*_n(Z_n^m, Y_n^m)$ matches exactly $Y_{n+1}^m$. In Section \ref{backpolicy} we show how to position training points for the inventory at time $n$ to avoid this resimulation step.

\subsection{Algorithm}

Algorithm \ref{alg:3} lists a pseudocode for value and performance iteration using regress-later approach.
Notice that the absence of interdependence among the training points $\{Z_n^m,Y_n^m\}_{n=1}^N$ at different times allows, at each time step, to place the training data where it is more convenient or needed. However, the optimal design of the training set is beyond the scope of the present paper and we generate observations $Y_n^m$ of $I_n$ from a uniform distribution on $[0,I_{\text{max}}]$, while $X_n^m$ are obtained from the simulation of the uncontrolled process $X$.

We would like to mention  a method for the flexible design of basis functions for which the conditional expectation $\hat\phi^n_k$ can be computed analytically. The method has been presented first in \citet{Bouchard2012} in the context of option pricing. It consists of dividing the state space into hypercubes and choosing affine functions within them. This approach allows to capture complex shapes as long as the number of hypercubes is sufficiently large. In our framework, this corresponds to choosing basis functions $\phi_k(x,i)= (a_k x + b_k i + c_k) 1_{\{(x,i)\in B_k\}},$ with disjoint hypercubes $B_k$.

\begin{algorithm}[tb]
 \caption{Regress-later Monte Carlo algorithm (RLMC)\label{alg:3}}
 \begin{algorithmic}[1]
\State Generate training points $\{Z_{n}^{m}\}_{n=1,m=1}^{N,M}$ accordingly to the dynamics induced by $X$
\State Choose training points $\{Y_{n}^{m}\}_{n=1,m=1}^{N,M}$
\State Initialise the value function $V (N, Z_{N}^{m},Y_{N}^{m}) =g(Z_{N}^{m},Y_{N}^{m}), \quad \forall m$
 \For{n=N-1 to 1}
\State  $\alpha_{n+1}=\argmin\limits_{\alpha\in\mathbb{R}^K} \big\{\sum_{m=1}^M [V(n+1,Z_{n+1}^{m},Y_{n+1}^{m})-\sum_{k=1}^K\alpha_k\phi_k(Z_{n+1}^{m},Y_{n+1}^{m})]^2 \big\}$ 
 \If{value iteration}
 \State For all $m$ do
\begin{flalign}\label{eqn:RLMC_alg_vi}
&V(n,Z_{n}^{m},Y_{n}^{m})\\ 
&=\sup_{\substack{u \in [\underline{u}, \overline{u}],\\ \varphi(j, u, Y_{j}^{ m}) \in [0, I_{\max}]}} \Big\{f(n,Z_{n}^{m},Y_{n}^{m},u)+\sum_{k=1}^K\alpha_{n+1}^k\hat \phi^n_k(Z_{n}^{m}, \varphi (n, u, Y_{n}^{m}))\Big\} \notag
\end{flalign}
\EndIf
 \If{performance iteration}
\State Define control mapping $u^*_n$ by formula \eqref{eqn:back_cont}
\State Apply \textbf{Algorithm 4}, or do the following:
\State\qquad Recalculate $\{Y_j^m \}_{j=n+1}^N$ for all $m$ using control mappings  $\{u^*_j\}_{j = n, \ldots, N-1}$
\State \qquad $V(n,Z_n^{m},Y_n^m)=\sum_{j=n}^{N-1} f\big(n,Z_j^m,Y_j^m, u^*_j (Z_j^m,Y_j^m)\big)+g(Z_N^m,Y_N^m) \quad \forall m$
\EndIf
\EndFor
\State \textbf{Evaluate the policy}
\end{algorithmic}
 \end{algorithm}

\subsection{Backward construction of inventory levels for performance iteration}
\label{backpolicy}

The main bottleneck in using performance iteration for inventory problems is the need to resimulate the inventory level until maturity after each backward step. When the algorithm moves closer to the initial time, the length of trajectories that need to be resimulated grows. Eventually, the computational effort of resimulation is proportional to $N^2$ compared to a linear cost of simulating the process $(X_n)$. Moreover for each step of the resimulated process an optimisation problem has to be solved in order to determine the policy associated to the new trajectory, heavily increasing the computational burden. In this section, we show how this can be avoided.

It follows from \eqref{eq:lspolicy} in Section \ref{value} that in order to avoid resimulation to compute $J(n,(Z_j^m, Y_j^m, u^*_j)_{j=n, \ldots, N})$, one needs to be able to reuse $J(n+1,(Z_j^m, Y_j^m, u^*_j)_{j=n+1, \ldots, N})$. Notice that an interpolation, analogously as in the grid discretisation approach, it is not a viable alternative because $J(n,(Z_j^m, Y_j^m, u^*_j)_{j=n, \ldots, N})$ can not be represented in terms of $(X_n,I_n)$ only as it depends on the whole future trajectory $(Z_j^m, Y_j^m)_{j=n, \ldots, N}$.

Here we propose a solution in which we aim to place points $Y_n^m$ in such a way that optimally controlled process reaches the state $(X_{n+1}^m, I_{n+1}=Y_{n+1}^m)$ at time $n+1$ and therefore the existing path $(Z_j^m, Y_j^m)_{j=n+1, \ldots, N}$ is a valid successor of $(Z_n^m, Y_n^m)$ enabling reuse of $J(n+1,(Z_j^m, Y_j^m, u^*_j)_{j=n+1, \ldots, N})$. 
Recall that the inventory process follows the dynamics
\[
I_{n+1}=\varphi(n,u^*_n(X_n, I_n),I_n).
\]
From the optimality condition \eqref{eqn:back_cont}, we obtain that the pathwise optimal antecedent of a point $Y_{n+1}^m$ is a solution $y \in [0, I_{\max}]$, if exists, of the equation
\begin{equation}
\label{eq:fixedpoint}
Y^m_{n+1} = \varphi\bigl(n, u^*_n (Z_n^m, y), y\bigr).
\end{equation}
In the following lemma, we explore the existence of such y. Define two quantities $\underline{\varphi_n}=\max\limits_{u\in[\underline{u},\overline{u}]}\{\varphi(n,u,0)\}$ and  $\overline{\varphi_n}=\min\limits_{u\in[\underline{u},\overline{u}]}\{\varphi(n,u,I_{\max})\}$ and assume $0\le \underline{\varphi_n}<\overline{\varphi_n}\le I_{\max},\,\forall\,n$.
\begin{lemma}\label{lem:1}
A solution $y \in [0,I_{\max}]$ to \eqref{eq:fixedpoint} exists if the following conditions are satisfied:
\begin{itemize}
\item[1.] $Y_{n+1}^m\in[\underline{\varphi_n},\overline{\varphi_n}]$,
\item[2.] the function $u^*_n(Z_n^m,y)$ in \eqref{eqn:back_cont} is continuous with respect to $y$.
\end{itemize}
Moreover, if the map $y \mapsto \varphi(n,u^*_n(Z_n^m,y),y)$ is strictly monotone then the solution is unique. 
\end{lemma}
\begin{proof}
Introduce the function $\psi(y)=Y_{n+1}^m-\varphi(n, u^*_n (Z_n^m, y), y)$ and notice that if condition 2. is satisfied the continuity of $\psi$ is enforced by $\varphi$. Evaluate $\psi$  it at the boundaries of $[0,I_{\max}]$:
\[
\begin{split}
\psi(0)=& Y_{n+1}^m-\varphi(n, u^*_n (Z_n^m, 0), 0) \\
\ge &\underline{\varphi_n}-\underline{\varphi_n}=0; \\
\psi(I_{\max})=& Y_{n+1}^m-\varphi(n, u^*_n (Z_n^m, I_{\max}), I_{\max})   \\
\le & \overline{\varphi_n}-\overline{\varphi_n}=0.
\end{split}
\]
Therefore, if condition 1. is satisfied it must exists a point $y \in [0, I_{\max}]$ such that $\psi(y)=0$, i.e., solve equation \eqref{eq:fixedpoint}.
In addition, if $y \mapsto \varphi(n,u^*_n(Z_n^m,y),y)$ is strictly monotone then so is $\psi$ and hence there exists only one $y$ which solves equation \eqref{eq:fixedpoint}.
\end{proof}

\begin{algorithm}[tb]
\caption{Fixed point performance iteration algorithm\label{alg:back}}
\begin{algorithmic}[1]
\State Let $y$ be a solution to $Y_{n+1}^m = \varphi\bigl(n, u_n (Z_{n}^{m}, y), y\bigr).$
\If{solution exists and $y\in[0,I_{\text{max}}]$}
\State Set $Y^m_{n}=y$  
\State $V(n,Z_{n}^{m},Y_{n}^{m})=f\big(Z_{n}^{m},Y_{n}^{m}, u^*_n (Z_{n}^{m}, Y_{n}^{m}) \big)+V\big(n+1,Z_{n+1}^{m}, \varphi(n,u^*_n (Z_{n}^{m}, Y_{n}^{m}),Y_{n}^{m})\big)$
\Else
\State Generate $Y_{n}^{m}$ randomly in $[0, I_{\max}]$
\State Resimulate to compute $V(n,Z_{n}^{m},Y_{n}^{m})$
\EndIf
%\EndIf
 \end{algorithmic}
 \end{algorithm}

Lemma \ref{lem:1} ensures the existence of solution to the equation \eqref{eq:fixedpoint} even though such solution could be found also for weaker assumptions, possibly outside the interval $[0, I_{\max}]$. Whenever this solution does not lie in $[0, I_{\max}]$, i.e. $Y_{n+1}^m\notin[\underline{\varphi_n},\overline{\varphi_n}]$ ,  it is not possible to determine a step backward $Y^m_n$ for the inventory variable. In such case, we propose to break the path and position the point $Y_n^m$ at random in $[0,I_{\text{max}}]$. Following this step, we resimulate the path $\{Y^m_j\}_{j=n+1, \ldots N}$ as in the traditional performance iteration algorithm. The random positioning of training points $Y^m_n$ when the solution of the constrained fixed point problem \eqref{eq:fixedpoint} cannot be found avoids the clustering of training points in certain areas of the state space.

The phenomenon described above that there is no solution to \eqref{eq:fixedpoint} that satisfies the bounds on the inventory level is not unusual in many practical examples and, in particular, does not result from approximation errors or any other features of the numerical scheme. In an example in the next section, we consider trading a commodity, with price $(X_t)$, subject to constraints on the size of storage. An optimal behaviour when the price is high is to sell. However, if this happens, an optimal trading would never lead to a full inventory at the next time period, $Y^m_{n+1} = I_{\max}$. In fact, with high commodity prices over a period of time and an optimal strategy of selling, the backward constructed trajectory $Y^m_j$ would climb up (when time goes backwards) finally hitting the level $I_{\max}$. 

\begin{xmpl}
It is commonly seen in the inventory problems that the control represents additive adjustments to the inventory level, i.e. $\varphi(n, u, i) = i + u$. Equation \eqref{eq:fixedpoint} simplifies to
\[
Y^m_{n+1} = y + u^*_n(Z^m_n, y).
\]
Denoting $\Delta y = Y^m_{n+1} - y$ we obtain a classical fixed point problem
\[
\Delta y = \hat \nu (\Delta y),
\]
where $\hat \nu(\Delta y) = u^*_n(Z^m_n, Y^m_{n+1} - \Delta y)$. Lemma \ref{lem:1} implies that exists an unique solution $\Delta y\in[\underline{u},\overline{u}]$ for every $Y_{n+1}^m\in[\overline{u},I_{\max}+\underline{u}]$.
\end{xmpl}

In the following section we evaluate our methods on two examples and compare them against their main competitors: grid discretisation and control randomisation.

\section{Numerical evaluation}
\label{numerics}

This section collects numerical results showcasing the performance of our regress-later approach.  First example is a ``price arbitrage'' problem in the context of energy trading for which we compare the accuracy and running time of grid discretisation,  control randomisation and the regress-later Monte Carlo approach in both value and performance iteration. %here everything works fine

The second control problem  features a two dimensional inventory in the context of managing a system of hydro reservoirs. The aim is to maximise the profit from selling the electricity produced by a set of turbines installed in the pipes connecting the reservoirs. The mathematical formulation is similar to the previous problem, however the higher dimension of the inventory highlights the impact of dimension on the computational complexity of  grid discretisation.

Our third control problem is multi-dimensional and consists in controlling the operations of a power system comprised of a wind turbine, a group of residential buildings, an energy storage and an interconnection with the main grid with the aim of maximising the profit over a finite horizon. Due to the dimensionality the underlying process and a multi-dimensional control, we apply regress-later algorithm only. The multi-dimensionality of the control makes control randomisation significantly slower than RLMC. %here we show when CR is slow

%conclusion: RL is more flexible

\subsection{Price arbitrage problem}
Consider the following specification of the dynamics of a commodity price as an Ornstein-Uhlenbeck process and inventory in continuous time (with time counted in years):
\begin{equation}
\label{eq:inventorybenchmark}
\begin{cases}
dX_t=2(5-X_t)\,dt+5\,dW_t, &\\
dI_t=u_t\,dt\,,\,\,\,u\in\mathcal{U}, &
\end{cases}
\end{equation}
where $\mathcal{U}=\{u_s\in \{-11.5,0,11.5\}: \ 0\le I_t \le 1 \,\,\forall t\}$ consists of adapted processes. The objective is to maximise the functional
\begin{equation}
\mathbb{E}\Bigl[\int_{t}^T -(u_s+\rho_{u_s})\,X_s\,ds+\frac{1}{2} X_N\,I_N |X_t=x,I_t=i\Bigr],
\end{equation}
where $\rho_{u_s}$ introduces a cost associated with each control:  $\rho_{u_s} =1_{\{u_s \ne 0\}}2$. 

We use an Euler discretisation of the dynamics and the performance measure for $T = 1$ and a time step $\Delta = 1/200$ yielding $N=200$ time periods. The resulting dynamic programming equation is of the form
\[
V(t,x,i)=\max_{\substack{u \in \{11.5, 0, -11.5\},\\ i + u\Delta \in [0, 1]}} \Bigl\{ -(u+\rho_{u})\,X_n \Delta +\mathbb{E}\bigl[V(n+1,X_{n+1}, i + u)|X_n=x\bigr]\Bigr\}.
\]

\subsubsection{Implementation}

All algorithms are implemented in MATLAB and are optimised in a similar fashion in order to reveal the relative differences in their computational complexity. In particular, we do not take advantage of vectorisation whenever this speed-up cannot be shared by all algorithms. The basis in regress later algorithms consisted of polynomials up to order two in $X$ and  $I$, i.e.~$\{  1,\, x,\, i,\, xi,\, x^2 ,\,i^2,\,x^2i,\linebreak\,x^2i^2\}$. Basis functions for control randomisation involve also the control variable, so we employed a basis comprising polynomials up to order two in $X$, $I$ and $u$, i.e.~$\{1,\,x,\,i,\,u,\, xi,\, x^2,\,i^2,\, u^2,\,xu,\,iu \}$. For the grid discretisation method we discretised the inventory in $L$ levels and then used polynomial basis functions up to order two for $X$, i.e. $\{  1,\, x,\, x^2 \}$. 

\subsubsection{Comparison}

We run our experiments with various numbers of simulated paths and grid points (for grid discretisation) or simulated paths (for control randomisation) or training points (for regress-later) and evaluate the resulting policies over a the same set of simulations shared between all methods and runs in order to minimise the effect of the Monte Carlo error on relative results. We record running time and value of estimated policies and display them in Figure \ref{fig:performance_time}.
The results for  the value iteration specification are shown on the left panel. Notice that the regress later is more efficient than all other methods; it can achieve similar levels of accuracy to the other algorithms  in less time.  Grid discretisation on the other hand is the slowest method, but it can nonetheless achieve high levels of accuracy. Control randomisation performs slower to regress later as it suffers from the need of fitting an higher number of basis functions. %We can note in particular that increasing the number of discretization points alone might only slow down computations without guaranteeing higher precision. 
 
 In the performance iteration specification, see the right panel of Figure \ref{fig:performance_time},  we only display two versions of regress-later algorithm: with and without backward construction of paths. Control randomisation and grid discretisation are less efficient as in the value iteration case with the latter requiring  disproportionally  long computational time due to multiple resimulations for each discretisation point.

Regress-later algorithm without resimulation of paths is slower but more accurate that the one with backward construction of paths. This improved accuracy may be explained by the resimulation present in the former while the latter uses the same training points for both estimation of the optimal policy and its evaluation introducing bias; this drawback has already been observed in literature in other contexts. The use of backward construction of the paths reduces, for this particular problem, the number of resimulated paths to about $25\%$ of the original value as approximately one quarter of paths are broken at each time step. The lower precision of the regress later with backward construction of paths combined with its lower computational complexity makes it useful only when the time constrain is particularly tight. We observe that for problems with continuous space of controls the improvement offered by this backward simulation method is considerable as the optimization problem to be solved at each time step of the resimulation procedure is more time consuming than in the present problem where controls can take one of three values. % We noticed that the performance of Control randomisation decline with the number of time steps considered, showing that this method is not reliable in its performance iteration specification. Similarly grid discretization is particularly sensitive to the resimulation, repeated at each discretization step, making the method few orders of magnitude slower than the alternatives.
  
 The difference in performance of estimated policies between the value and performance iteration observed in this example might be explained as follows. Firstly, the performance iteration problem has been slightly modified considering shorter timesteps $\Delta=1/1000$ rather than $\Delta=1/200$ in order to make the resimulation step more burdensome and highlight the differences between the algorithms with and without backward construction of paths. Secondly, the functional form of terminal condition and running reward make this optimisation problem linear quadratic over most of the domain (non-linearities are encountered at the boundaries of the inventory) indicating that basis function can replicate well the true value function, limiting the propagation of the error. In these cases the use of performance iteration is rarely advised as it implies higher variance of the points used to fit the regression parameters, reducing the estimation precision.
  
 % It is easy to deduce from the algorithms that the computational complexity with respect to the number of time periods $N$ is quadratic for methods with path resimulation (CRP, RLMCP, RLMC BP) and linear for others (RLMC BV, GD). 
  \iffalse
  \begin{figure}[H]
\centering
\includegraphics[width=0.9\linewidth]{Value_time.jpg}
\caption{The figure above shows the convergence of the three algorithms in value iteration specification. Note how, under one second of computational time, a clear ordering of the algorithm can be observed: grid discretization , control randomisation and regress later; in order of increasing performance. On the other hand it can be noticed that when the computational budget increases the three methods perform similarly and only the variance of the estimation is different.} 
\label{fig:value_time}
\end{figure}

  \begin{figure}[H]
\centering
\includegraphics[width=0.9\linewidth]{Performance_time.jpg}
\caption{blue RL, violet RLBP, green CR, orange GD } 
\label{fig:performance_time}
\end{figure}
\fi

  \begin{figure}
\centering
\includegraphics[width=\linewidth]{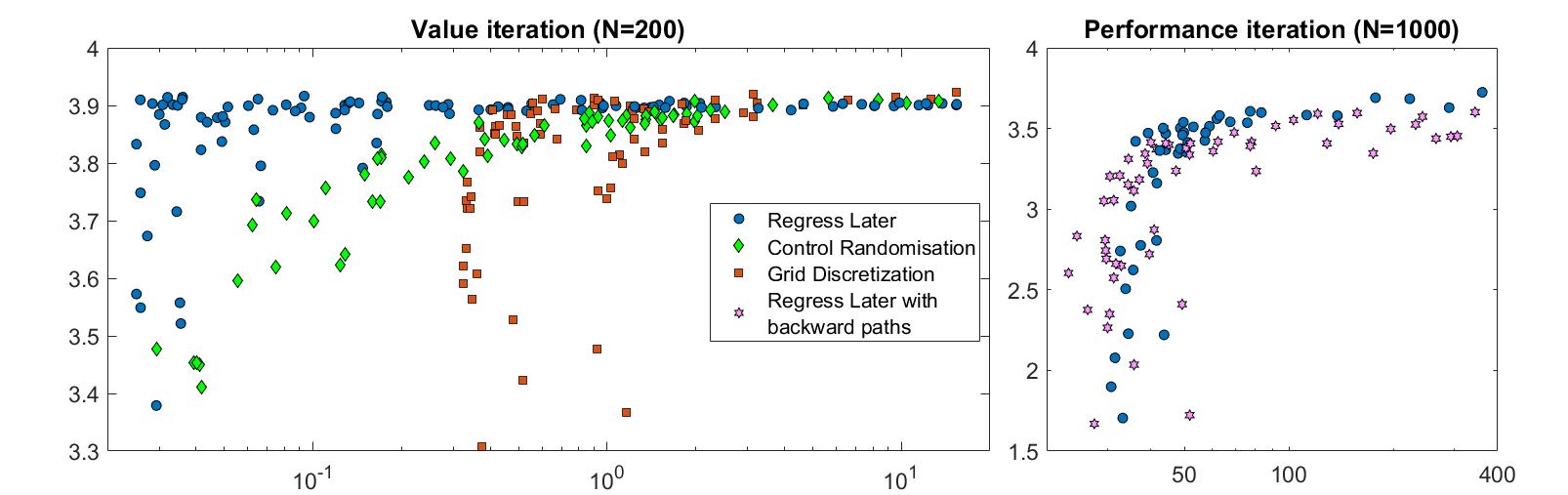}
\caption{Trade-off time versus performance of estimated policies for four algorithms studied in this paper. The horizontal axis shows time. The vertical axis displays the performance of estimated policies computed using a fixed set of 1000 paths to aid comparison. Value iteration algorithms with time discretisation $\Delta = 1/200$ are shown on the  left panel. The right panel displays results for performance iteration and $\Delta = 1/1000$ to amplify the effect of resimulation.}
\label{fig:performance_time}
\end{figure}

\subsection{System of pumped hydro-reservoirs}\label{subsec:2D}

In this section we extend the previous problem to a case in which the inventory is two dimensional. We consider a problem of  managing a system of two interconnected hydro reservoirs with turbines and pumps, see Figure \ref{fig:reser}. Similarly to the previous case, we model the electricity price as AR(1) process:
 \[
 X_{n+1}=X_{n}+\alpha(\mu-X_n)+\sigma\xi_n, \qquad \xi_n\sim\mathcal{N}(0,1),\,i.i.d.,
 \]
where one time step represents 30 seconds. We aim to profit from buying and selling electricity. We use the parameters $\sigma=1$, $\alpha=0.1$ and $\mu=40$.
 \begin{figure}[tb]
\centering
\includegraphics[width=0.5\linewidth]{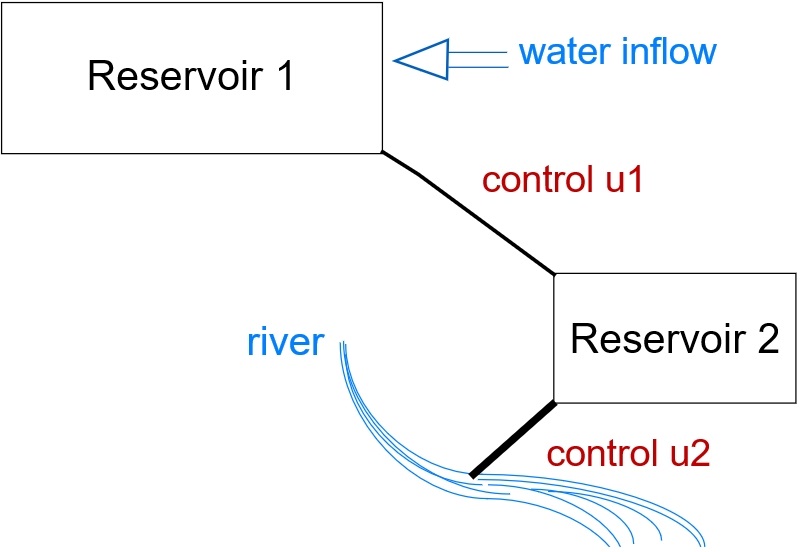}
\caption{Interconnected system of reservoirs.} \label{fig:reser}
\end{figure}

Reservoir 1 is fed with a constant water inflow equivalent to 120MW (this assumption is reasonable given the 3 hours horizon). The  connections between Reservoir 1 and Reservoir 2 is equipped with a pump and a turbine with production/consumption capacity of 600 MW (for simplicity of presentation but without any loss of numerical difficulty we assume 100\% round-trip efficiency). The connection of Reservoir 2 with the river can accommodate a higher flow of water and contains a set of turbines with maximum generation capacity  1.2 GW. The size of Reservoir 1 is 2 GWh, the size of Reservoir 2 is 1 GWh. The numbers chosen for this example have been inspired by the facilities installed in Dinorwig, UK, and in multiple locations in Norway.

The dynamics of the energy stored in the reservoirs follow::
\[
I_{n+1}^1=I_{n}^1+u^{1}_n+0.12,
\]
\[
I_{n+1}^2=I_{n}^2-u^{1}_n+u^{2}_n,
\]
where $u^1_n \in [-0.6, 0.6]$ is the flow between Reservoir 1 and 2, and $u^2_n \in [0, 1.2]$ is the flow from Reservoir 2 to the river. We define  the performance measure over three hours time horizon as follow:
\[
J(u^{12},u^{2r})=\sum_{s=1}^N -X_s(u^{12}_n+u^{2r}_n)-A(I^{1}_N+I^{2}_N-1.5)^2
\]
where the terminal condition encourages that the total energy stored in both reservoirs at terminal time equals 50\% of the maximum of 3 GWh. %This condition will force the controller to release water in the river only when it is worth the cost.
An example of the optimal strategy estimated by Regress Later is displayed in Figure \ref{fig:reservoirs}
\begin{figure}[tb]
\centering
\includegraphics[width=\linewidth]{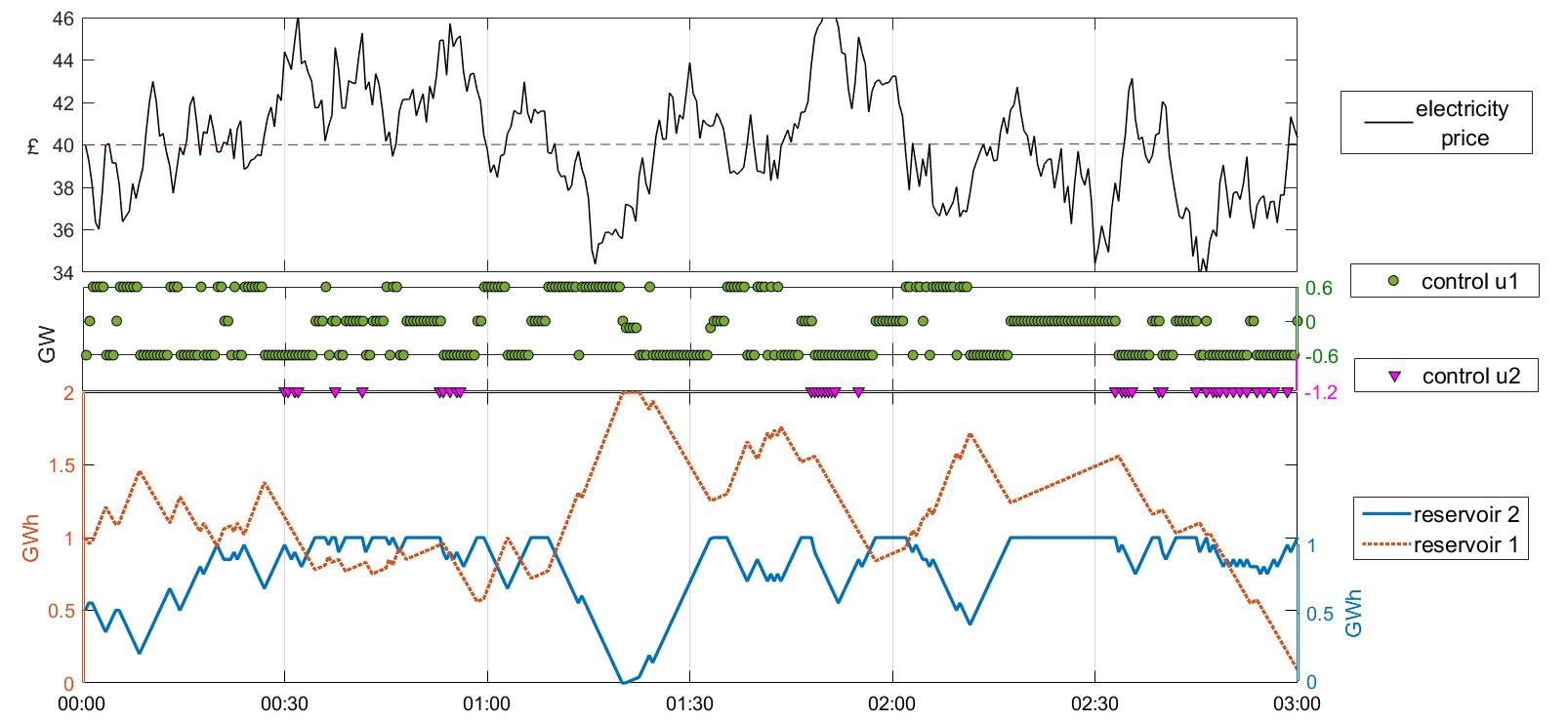}
\caption{At the top the electricity price. In the middle the control, $u^{12}$ in green and $u^{2r}$ in violet. At the bottom in blue inventory 2 (1 GWh capacity) and in orange inventory 1 (2 GWh capacity)} 
\label{fig:reservoirs}
\end{figure}

We extended our implementation of grid discretisation to this problem to show the effect of the dimension of inventory on its performance. Consider for example that if we take 7 discretisation points for each storage, we obtain a total of 49 possible configurations, for which the estimation of regression coefficients and computation of the optimal control has to be repeated. This leads to significantly lower computational efficiency compared to regress-later algorithm, see Table \ref{tab:} which lists running time and values implied by the estimated policies for regress-later and grid discretisation.

\begin{table}[tb]
\centering
\caption{Performance of estimated policies evaluated on a fixed set of 1000 paths. For comparison, the myopic policy's performance is $274$.}
\label{tab:}
\begin{tabular}{l|ccc|ccc}
 &  \multicolumn{3}{c|}{Regress later} & \multicolumn{3}{c}{Grid discretisation } \\ \hline \hline
 value & 340 & 355 & 357 & 290 & 328 & 357\\
 time & 0.2 s & 0.5 s & 6 s & 4 s & 7.5 s & 35 s \\
 \end{tabular}
\end{table}

\subsection{Energy Storage management}\label{subsec:energy_stor}

In this subsection we present a high dimensional problem in which the control is initially five-dimensional and the state space four-dimensional; we will reduce the action space to two dimensions and solve the problem applying $RLMC$ algorithm.

For the model in this example we take inspiration from \citet{Powell2014} and consider an electricity system composed of an intermittent energy supply represented by a wind turbine, a load or electricity demand represented by a group of residential buildings, a battery storage device and an interconnection with the main grid. Energy can flow through the different components of the system as indicated by the arrows in Figure \ref{fig:powell}; note that the power produced from the wind turbine is directly transferred to the load and used to satisfy the demand, the power in excess is transferred to the storage device which can sell it to the main grid or save it for later use in the battery. When the power from the wind turbine is insufficient to satisfy the demand, the shortfall is fulfilled either (or both) from the storage device discharging the battery or/and by buying electricity from the grid. Note that we assume that the main grid can always provide the power we require.

The operator has control on every flow of energy, making the dimensionality of the control process equal to five; we denote these flows at time $n$  by $X^{WD}_n$,$X^{WB}_n$,$X^{BD}_n$,$X^{GB}_n$ and $X^{GD}_n$, where the notation $X^{i,j}$ indicates energy flowing from $i$ to $j$ and we denote the components as follow: $W$ for the wind turbine, $D$ for the demand, $G$ for the grid interconnection and $B$ for the battery. We constrain each flow to be positive except for $X^{GB}_n$ which can assume negative values indicating that the battery is selling energy to the grid. Finally we consider our problem stated on an hourly basis over a time horizon of two weeks. %We explain in detail modelling assumptions in Section \ref{sec_app:wind} in Electronic Appendix and provide here only necessary details.

As in \citet{Powell2014} we make the following modelling choices.

\subsubsection*{Wind Power}
We compute the power in MWh produced by a wind turbine with rated power $2.1\,MW$ in one hour as a function of the wind speed $w_n$ (average over an hour in m/s):
\[
W_n=10^{-6}\,C_p\,\rho\, A \,\max(w_n,\omega)^3,
\] 
where $C_p=0.4$ is the power efficiency, $\rho=1.225\,kg/m^3$ is the density of air, $\omega=14 m/s$ represents the rated output speed and  $A=50^2\pi$ is the area swept by the blades.

In particular we model the centred square root of the wind speed (which we consider constant) as an AR(1) process:
\[
\begin{split}
w_n=&\left( y_n+\mu\right)^{2}\\
y_{n+1}=&0.7633 y_{n}+0.4020\xi_n,\quad\xi_n\sim\mathcal{N}(0,1)\\
\end{split}
\]
The coefficients are calibrated from 2010 real data registered in Maryneal, Texas; see \citet{Powell2014} for details. 
\subsubsection*{Electricity Price}
The electricity price is modelled as is \citet{Powell2014} as an exponential of a mean reverting process superimposed on a weekly seasonal component $Y^{week}$ estimated from the data recorded in Texas in 2010:
\[
\begin{split}
P_{n+1}=&e^{Y_{n+1}}-27.2531+Y^{week}_{n+1}\\
Y_{n+1}=&Y_n+0.2055(4.1995-Y_{n})+0.11856 \xi_n+J_n,\quad \xi_n\sim\mathcal{N}(0,1)
\end{split}
\]
where $J_n=\xi_n^j1_{\{U_n<0.017\}}$ is the jump component, $\xi_n^j\sim\mathcal{N}(0,0.4229)$ and $U_n\sim\mathcal{U}(0,1)$.

Note that we introduce a fixed $10\%$ transmission surcharge on the use of the grid for both supplying demand and exchanging energy with the battery. This is a requirement in order to make the optimal control unique, otherwise it is equivalent to supply power to the load from the battery through $X^{BD}$ or through $X^{GB}\to X^{GD}$.  
\subsubsection*{Electricity Demand}  
It is well known that temperature is one of the main drivers of the electricity demand of residential buildings, in fact during periods of high or low temperature households tend to use air conditioning.
The electricity demand is therefore modelled as a function of temperature; in \citet{Powell2014} it is suggested to use an order 5 polynomial to fit the observed pairs of demand and temperature. Since estimated coefficients were not provided and in order to be consistent with the other data, we estimated our coefficients using observations recorded in Austin, Texas in 2010 (\url{http://www.noaa.gov/}). We obtained the following model for the demand:
\[
\tilde{D}_{n}=6784.9728-235.5911\,T_{n}-2.2869\,T^2_{n}+0.8897\,T^3_{n}-0.0204\,T^4_{n}+0.000105\,T^5_{n}.
\]
To scale down the demand to a few buildings we set $D_{n}=\frac{\tilde{D}_{n}}{30000}$. 
We model the temperature as a simple mean reverting process $T_n=\tilde{T}_n+\mu_n$, where $\mu_n$ represents a daily seasonal component estimated from the data and
\[
\tilde{T}_{n+1}=-0.92\tilde{T}_n+2.14\xi_n, \quad \xi_n\sim\mathcal{N}(0,1).
\]
\subsubsection*{The Battery}
The most interesting component of the system is an energy storage device; it is represented by a battery with rated capacity $I_{\max}=20$MWh and maximum power in charge $\Pi_{\max}=2.1$MW and discharge $\Pi_{\min}=-2.1$MW. For simplicity of exposition we assume the battery is $100\%$ efficient and its use does not take into account its ageing. The level of charge of the battery evolves as follow:
\[
I_{n+1}=I_n+X^{GB}_n+X^{WB}_n-X^{BD}_n.
\]

\begin{figure}
\centering
\includegraphics[width=1\linewidth]{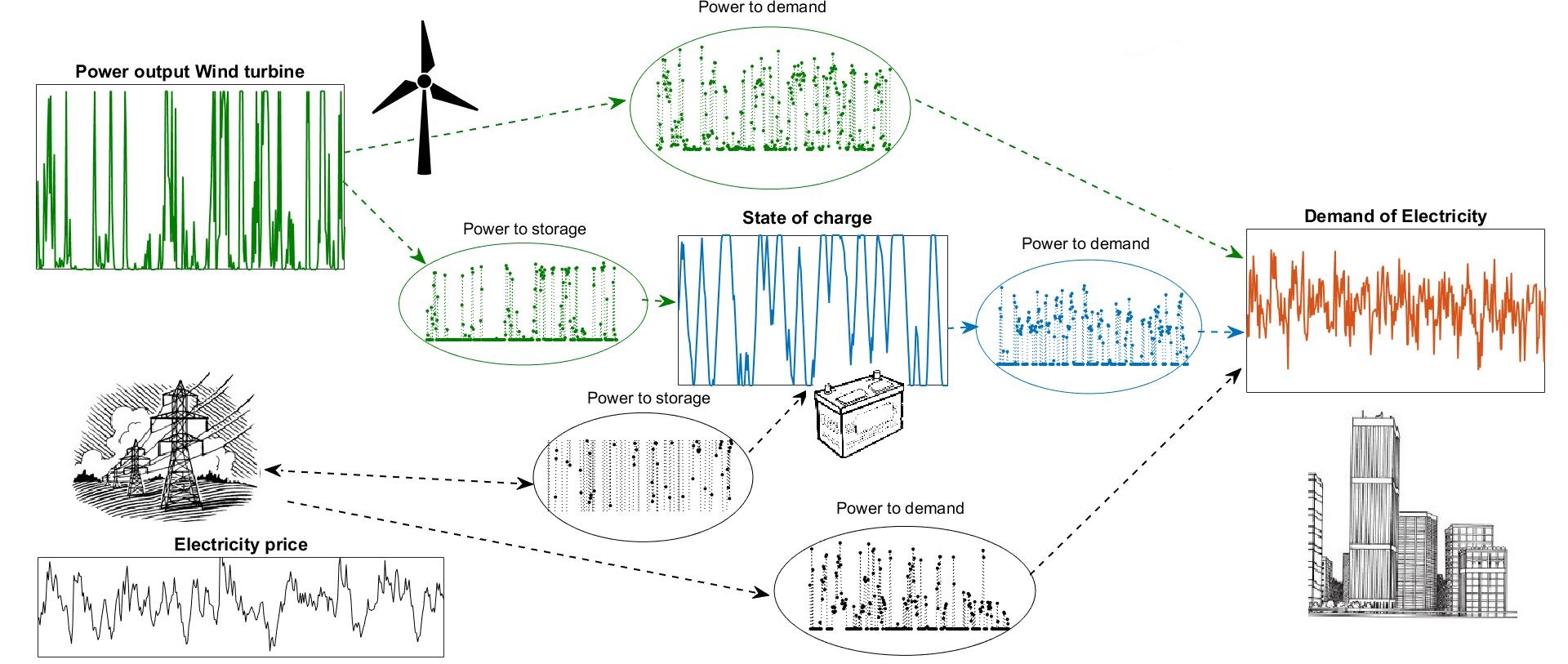}
\caption{The diagram above shows components of the system. Boxes show evolution of state variables. Ellipses display energy flows among the components. Note that the algorithm directly controls the connection grid-battery and grid-demand as other flows are uniquely determined by those two. }
\label{fig:powell}
\end{figure}

\subsubsection{Problem formulation and solution}
We look at the system on an hourly basis during the first two weeks of January, which corresponds to $336$ hourly time steps. We aim to provide power to the residential buildings minimising the cost of buying electricity from the grid.  Since the problem is stated in discrete time we will deal with energy rather than power. Notice that the following set of constraints is in place (we dropped the subscript $n$):
\begin{align}
& X^{GD}+X^{WD}+X^{BD}=D \label{demand_cons}\tag{c.1} \\
 &X^{WD}=\min\{W,D\} \label{windem_cons}\tag{c.2} \\
& X^{WB}+X^{WD}=W \label{wind_cons}\tag{c.3} \\
% &X^{GB}X^{BD}\le 0 \quad X^{WB}X^{BD}\le 0 \quad X^{WB}X^{GB}\ge 0 \label{cons} \tag{c.4} \\ 
&X^{GD}\ge 0 \label{griddem_cons} \tag{c.5}\\
&X^{BD}\in[0,-\Pi_{\min}] \label{battdem_cons}\tag{c.6} \\
&X^{GB}+X^{WB}-X^{BD}\in[\Pi_{\min},\Pi_{\max}]\cap[-I,I_{\max}-I] \quad \quad \label{inventory_cons}\tag{c.7}
\end{align}

The first constraint \eqref{demand_cons} is natural and says that the demand has to be matched at all times.
Constraints \eqref{windem_cons} and \eqref{wind_cons} regulate the operations of the wind turbine: all its power is delivered to the load first and the excess is transferred to the battery. As long as the current level of charge allows it the power is accumulated; the excess power is sold to the grid (for a price reduced by the $10\%$ transmission surcharge), c.f. \eqref{inventory_cons}.
Constraint \eqref{griddem_cons} enforces the natural condition that power should only flow from the grid to the load. Constraint \eqref{battdem_cons} limits the discharging rate of the battery enforcing a bound on the energy flowing from the battery to the demand. Finally the last constraint \eqref{inventory_cons} guarantees that the inventory stays between empty and full and that the maximum charging/discharging rates are maintained.

Note that the constraints \eqref{demand_cons}-\eqref{inventory_cons} actually reduce the dimension of the control process from five to two ($X^{GB}$ and $X^{GD}$), as the other components can be computed consequently. Recall that $X^{GD} \ge 0$ while the sign of $X^{GB}$ depends on whether the battery buys or sells energy to the grid. 

We can then write the profit of an electricity provider who owns the wind turbine and the battery and has contracted a given load, as electricity sold minus electricity bought (net of transmission surcharge):
\[
J\left(\left(D_s,W_s,P_s,B_s,I_s,X^{GB}_s,X^{GD}_s\right)_{s=1,\ldots, N}\right)=\sum_{n=1}^N P_n\,D_n-1.1P_n(X^{GB}_n+X^{GD}_n).
\]
Define the value function of the problem as the solution to the minimisation problem:
\begin{equation}
\label{eq:powell}
\begin{split}
V(n,&d,w,p,i)=\\ &\min_{X\in\mathcal{U}}\bigg\{-\mathbb{E}\Big[ J\big((D_s,W_s,P_s,B_s,I_s,X^{GB}_s,X^{GD}_s)_{s=n,\ldots, N}\big)|(D_n,\,W_n,\,P_n,\,I_n)=(d,w,p,i)\Big] \bigg\},
\end{split}
\end{equation}
where $\mathcal{U}$ is the set of admissible controls specified by constraints  \eqref{demand_cons}-\eqref{inventory_cons}.

\subsubsection{Numerical solution}
We solve the problem \eqref{eq:powell} applying the Regress Later algorithm introduced in Section \ref{later}. For this problem the choice of linear and quadratic basis functions, i.e. $\{D,\,W,\,P,\, D^2,\,W^2,\,P^2,\, I,\,D\,I,
\linebreak\, P\,I,\,W\,I,\,I^2\}$, works satisfactorily. We could also not see any significant improvement in results when using more than $5000$ simulated trajectories. The numerical complexity of the algorithm stems mainly from a non-trivial two dimensional constrained optimisation problem for the control which is solved at each time step and for each training point.

\begin{figure}[tbp]
\centering
\includegraphics[width=0.9\linewidth]{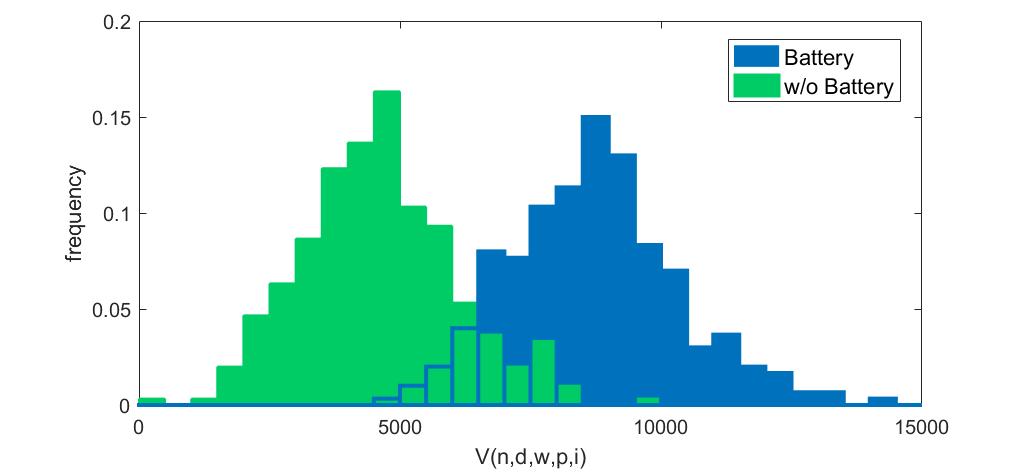}
\caption{The empirical probability distributions of the performance $J$ at time $0$ with the initial data $(d,w,p,i)=(0.15,0,35,10)$, representing the profits obtained in the system with (blue) and without (green) the use of the battery. There are 30 bins each 500 units wide.}
\label{fig:dwp_hist}
\end{figure}
\begin{figure}[tbp]
\centering
\includegraphics[width=0.9\linewidth]{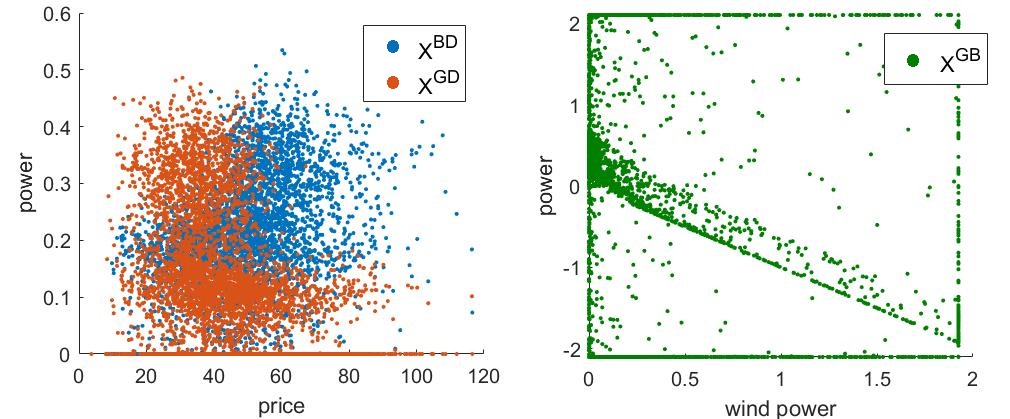}
\caption{This graphs display interactions within the system in the form of scatter plots. The left plot shows $X^{GD}$ and $X^{BD}$ against the price revealing that the first tends to provide power when price is low, the latter when price is high. The right plots show how the connection battery-grid ($X^{GB}$) is influenced by the wind power; note that the high level of wind often results in selling energy to the grid in order to decrease the level of charge.}
\label{fig:scatter_dwp}
\end{figure}

Once the estimated optimal policy has been computed we run it forward in time to assess its value; in Figure \ref{fig:powell} we display a particular realisation of all the flows of energy among the components. It can be observed that the battery stores the wind power when this exceeds the demand of electricity and transfers it to the demand instead of buying electricity from the grid at high price. Over the two weeks of interest the system considered has an average demand of electricity of $320$ KW compared with an average production of wind power of $350$ KW meaning that the system is a net exporter of electricity. Figure \ref{fig:scatter_dwp} displays some of the interactions within the system. The scatter plot on the left shows the energy flowing to the demand from the battery (blue) and from the grid (orange). There is a general tendency for the power to be bought from the grid when its price is low while it is supplied mainly by the battery when the price is high. The scatter plot on the right displays the interactions between wind power and transfers through the connection battery-grid. We observe that when the wind power is high the battery tends to sell energy to the grid in order to maintain an acceptable level of inventory.

The solution of the optimisation problem of this subsection enables study of  the economic value of using a battery to assist the operations of a wind turbine. To this end, we evaluate the estimated policy on a new and larger set of trajectories of the underlying processes (wind, demand and price). This is further compared to the profits made in a system with no battery installed, i.e., all excess wind energy immediately sold to the grid. We find that the $25\%$ and $75\%$ quantiles of the profits for operating the system without the use of the storage over two weeks are $3595$, and $5498$, respectively. When the flexibility provided by the storage is optimally used the respective quantiles are $7556$ and $9598$, showing an increase in the profits of around $4000$. Further details about the distribution of profits are displayed in Figure \ref{fig:dwp_hist}. Considering the inventory behaviour displayed in Figure \ref{fig:powell} as ``typical'' we can compute the total throughput of energy of $490$ MWh in the two weeks. Assuming further that the battery sustains $4000$ full cycles we obtain an estimated life-span of $320$ weeks or approximately $6$ years. Assuming further a stationary flow of profit during the life of the battery we conclude that the installation of the storage alongside the wind turbine will bring additionally $\$ 640000$ in profits (assuming zero scrap value and zero interest rate).

Our solution of problem \eqref{eq:powell} shows that RLMC is well suited for dealing with the optimal control of complex systems and therefore provides an alternative to  heuristics and policies obtained under perfect foresight used commonly for the solution of real world high dimensional problems.

\section{Conclusions}

We presented Monte Carlo methods for the optimal control of an inventory subject to a measure of performance based on an exogenous random process. Our main contribution was in designing numerical optimisation algorithms based on regress later approach and using value and performance iteration schemes. In particular, we showed how, solving a fix point problem, one can avoid resimulation of the inventory trajectories in performance iteration during the backward recursion steps. Numerical tests demonstrated that our regress-later Monte Carlo methods perform well compared to competitors and are suitable for solving real-world high-dimensional problems. 

%\section*{References}

%\bibliographystyle{plainnat}
\bibliographystyle{apalike}
\bibliography{reference}

% \clearpage
% \setlength{\baselineskip}{0.7\baselineskip}
% 
% \begin{center}
% \huge
% Electronic Appendix
% 
% \medskip
% \Large
% ``Regress-Later Monte Carlo for Optimal Inventory Control with applications in energy''\\
% by A. Balata and  J. Palczewski 
% \end{center}
% 
% \begin{appendices}
% \setcounter{table}{0}
% \setcounter{figure}{0}
% \setcounter{equation}{0}
% \setcounter{algorithm}{0}
% \makeatletter
% \@addtoreset{table}{section}
% \makeatother
% \setcounter{page}{1}
% \renewcommand{\thetable}{\thesection.\arabic{table}}
% \renewcommand{\thefigure}{\thesection.\arabic{figure}}
% \renewcommand{\theequation}{\thesection.\arabic{equation}}
% \renewcommand{\thealgorithm}{\thesection.\arabic{algorithm}}
% 
% 
% \end{appendices}

\end{document}